\documentclass{amsart}
\usepackage{amsfonts}

\theoremstyle{plain}
\newtheorem{theorem}{Theorem}

\newtheorem{corollary}{Corollary}
\newtheorem{proposition}{Proposition}

\theoremstyle{definition}
\newtheorem{definition}{Definition}

\newtheorem{example}{Example}

\theoremstyle{remark}
\newtheorem{remark}{Remark}

\numberwithin{equation}{section}

\begin{document}
\title[Constructions of Semi-Symmetric Frenet Curves]{Constructions of
Frenet Curves with respect to Semi-Symmetric Metric Connection}
\author{\c{S}aban G\"{u}ven\c{c}}
\address[\c{S}. G\"{u}ven\c{c}]{ Balikesir University, Balikesir, T\"{U}RK%
\.{I}YE}
\email[\c{S}. G\"{u}ven\c{c}]{sguvenc@balikesir.edu.tr}

\begin{abstract}
Using a semi-symmetric metric connection, we construct Frenet frame and
curvatures of curves in 3-dimensional manifolds and give examples of
semi-symmetric Frenet curves in Euclidean, Sasakian and Kenmotsu
manifolds.\medskip
\end{abstract}

\maketitle

\section{Introduction}

In a Riemannian manifold $\left( M^{n},g\right) $, a linear connection $%
\widetilde{\nabla }$ is called a semi-symmetric metric connection if its
torsion tensor field $T$ satisfies%
\begin{equation*}
T\left( X,Y\right) =\omega (Y)X-\omega (X)Y
\end{equation*}%
and%
\begin{equation*}
\left( \widetilde{\nabla }_{Z}g\right) \left( X,Y\right) =0,
\end{equation*}%
where $\omega $ is a $1$-form and $X,Y,Z\in \chi (M)$. Let $\nabla $ denote
the Levi-Civita connection on $M$. Then, the semi-symmetric metric
connection has the form%
\begin{equation}
\widetilde{\nabla }_{X}Y=\nabla _{X}Y+\omega (Y)X-g\left( X,Y\right) U,
\label{ssmc}
\end{equation}%
where $U$ is a vector field associated to $\omega $, i.e. $\omega
(X)=g\left( X,U\right) $, $\forall X\in \chi (M)$ \cite{Yano}.

This study consists of six sections. The first section is Introduction. In
Section \ref{sect-semisymmetric Frenet}, we introduce semi-symmetric Frenet
curves in 3-dimensional manifolds and give necessary definitions. Then we
find the necessary and sufficient conditions for semi-symmetric Frenet
curves to be semi-symmetric geodesics of $\mathbb{E}^{3},$ $\mathbb{R}%
^{3}(-3)$ and $\mathbb{H}^{3}(-1)$ in Sections \ref{sect-geodesicE3}, \ref%
{sect-geodesic-R3(-3)} and \ref{sect-geodesic-H3(-1)}, respectively. The
final section is reserved for the examples of semi-symmetric Frenet curves
in these manifolds.

\section{\label{sect-semisymmetric Frenet}Semi-Symmetric Frenet Curves in $3$%
-Dimensions}

In this section, we will define Frenet curves in $3$-dimensional Riemannian
manifolds endowed with a semi-symmetric metric connection.

Let $\gamma :I\rightarrow M$ be a unit-speed smooth curve, where $M=\left(
M,g\right) $ is a $3$-dimensional Riemannian manifold endowed with a
semi-symmetric metric connection $\widetilde{\nabla }$. Let $T=\gamma
^{\prime }$ denote the unit tangent vector field along the curve $\gamma $.
Firstly, if we differentiate $g(T,T)=1$ with respect to $\widetilde{\nabla }$%
, since it is a metric connection, we find $g(\widetilde{\nabla }_{T}T,T)=0$%
. So $\widetilde{\nabla }_{T}T\perp T.$

\begin{definition}
Let 
\begin{equation*}
\widetilde{\kappa }=\left\Vert \widetilde{\nabla }_{T}T\right\Vert =\sqrt{%
g\left( \widetilde{\nabla }_{T}T,\widetilde{\nabla }_{T}T\right) }.
\end{equation*}%
We call $\widetilde{\kappa }$ as \textit{semi-symmetric curvature of} $%
\gamma $. If $\widetilde{\kappa }=0$ (i.e. $\widetilde{\nabla }_{T}T=0$),
then $\gamma $ is called a \textit{semi-symmetric geodesic}. Semi-symmetric
geodesics are called \textit{semi-symmetric Frenet curves of order} $1$.
\end{definition}

\begin{definition}
If $\widetilde{\nabla }_{T}T\neq 0$, then 
\begin{equation*}
N=\frac{\widetilde{\nabla }_{T}T}{\widetilde{\kappa }}
\end{equation*}%
is called \textit{semi-symmetric normal of} $\gamma $.
\end{definition}

Notice that $g(N,N)=1$. Then 
\begin{equation*}
\widetilde{\nabla }_{T}g(N,N)=0,
\end{equation*}%
that is $g(\widetilde{\nabla }_{T}N,N)=0.$ Thus $\widetilde{\nabla }%
_{T}N\perp N.$

\begin{definition}
If $\widetilde{\nabla }_{T}N$ is parallel to $T$, we call $\gamma $ a 
\textit{semi-symmetric Frenet curve of order }$2$.
\end{definition}

Let $\gamma $ be a semi-symmetric Frenet curve of order $2$. $\widetilde{%
\nabla }_{T}N=\lambda T$ for some function $\lambda .$ Since $g(T,N)=0$, we
find%
\begin{equation*}
g(\widetilde{\nabla }_{T}T,N)+g\left( T,\widetilde{\nabla }_{T}N\right) =0,
\end{equation*}%
\begin{equation*}
g\left( \widetilde{\kappa }N,N\right) +g\left( T,\lambda T\right) =0,
\end{equation*}%
\begin{equation*}
g\left( T,\widetilde{\nabla }_{T}N\right) =\lambda =-\widetilde{\kappa }.
\end{equation*}%
As a result

\begin{proposition}
If $\gamma $ is a semi-symmetric Frenet curve of order $2$, then $\widetilde{%
\nabla }_{T}N=-\widetilde{\kappa }T.$
\end{proposition}

\begin{definition}
If $\widetilde{\nabla }_{T}N$ is not parallel to $T$, we call $\gamma $ a 
\textit{semi-symmetric Frenet curve of order }$3$.
\end{definition}

Let $\gamma $ be a semi-symmetric Frenet curve of order $3$. Since $g\left( 
\widetilde{\nabla }_{T}N,N\right) =0,$ $\widetilde{\nabla }_{T}N$ can be
written as%
\begin{equation}
\widetilde{\nabla }_{T}N=g\left( \widetilde{\nabla }_{T}N,T\right) T+g\left( 
\widetilde{\nabla }_{T}N,B\right) B  \label{nablaTN1}
\end{equation}%
for some unit vector field $B,$ which is g-orthogonal to $T$ and $N$. $B$ is
called the \textit{semi-symmetric binormal of} $\gamma $. From $g(T,N)=0$,
we have 
\begin{equation*}
g\left( \widetilde{\nabla }_{T}N,T\right) =-\widetilde{\kappa },
\end{equation*}%
so we can write%
\begin{equation*}
g\left( \widetilde{\nabla }_{T}N,B\right) B=\widetilde{\nabla }_{T}N+%
\widetilde{\kappa }T.
\end{equation*}

\begin{definition}
For a semi-symmetric metric Frenet curve of order $3$, we define 
\begin{equation*}
\widetilde{\tau }=\left\Vert \widetilde{\nabla }_{T}N+\widetilde{\kappa }%
T\right\Vert >0,
\end{equation*}%
which we call \textit{semi-symmetric torsion of }$\gamma $.
\end{definition}

As a result, from (\ref{nablaTN1}), for a semi-symmetric metric Frenet curve
of order $3,$ 
\begin{equation*}
\widetilde{\nabla }_{T}N=-\widetilde{\kappa }T+\widetilde{\tau }B.
\end{equation*}

\begin{proposition}
For a semi-symmetric metric Frenet curve of order $3$, 
\begin{equation*}
\widetilde{\nabla }_{T}B=-\widetilde{\tau }N.
\end{equation*}
\end{proposition}

\begin{proof}
Since $M$ is $3$-dimensional, we $\widetilde{\nabla }_{T}B\in span\left\{
T,N,B\right\} $, which is a $g$-orthonormal vector basis along the curve $%
\gamma $ by construction. So we can write%
\begin{equation}
\widetilde{\nabla }_{T}B=g\left( \widetilde{\nabla }_{T}B,T\right) T+g\left( 
\widetilde{\nabla }_{T}B,N\right) N+g\left( \widetilde{\nabla }%
_{T}B,B\right) B.  \label{0}
\end{equation}%
Differentiating $g\left( T,B\right) =0,$%
\begin{equation*}
g\left( \widetilde{\nabla }_{T}T,B\right) +g\left( T,\widetilde{\nabla }%
_{T}B\right) =0,
\end{equation*}%
\begin{equation*}
g\left( \widetilde{\kappa }N,B\right) +g\left( T,\widetilde{\nabla }%
_{T}B\right) =0,
\end{equation*}%
\begin{equation}
g\left( T,\widetilde{\nabla }_{T}B\right) =0.  \label{1}
\end{equation}%
Differentiating $g\left( N,B\right) =0,$%
\begin{equation*}
g\left( \widetilde{\nabla }_{T}N,B\right) +g\left( N,\widetilde{\nabla }%
_{T}B\right) =0,
\end{equation*}%
\begin{equation*}
g\left( -\widetilde{\kappa }T+\widetilde{\tau }B,B\right) +g\left( N,%
\widetilde{\nabla }_{T}B\right) =0,
\end{equation*}%
\begin{equation}
g\left( N,\widetilde{\nabla }_{T}B\right) =-\widetilde{\tau }.  \label{2}
\end{equation}%
Differentiating $g\left( B,B\right) =1,$%
\begin{equation}
g\left( \widetilde{\nabla }_{T}B,B\right) =0.  \label{3}
\end{equation}%
If we write (\ref{1}),(\ref{2}) and (\ref{3}) in (\ref{0}), we get $%
\widetilde{\nabla }_{T}B=-\widetilde{\tau }N.$
\end{proof}

Now,

\begin{definition}
We call $\left\{ T,N,B\right\} $ the \textit{semi-symmetric Frenet frame
field of} $\gamma $ if its order is $3$. For order $2$, $\left\{ T,N\right\} 
$ will be called the s\textit{emi-symmetric Frenet frame field of} $\gamma .$
Semi-symmetric Frenet curves of order $1$ (semi-symmetric geodesics) have
the s\textit{emi-symmetric metric Frenet frame field }$\left\{ T\right\} .$
\end{definition}

\begin{definition}
A semi-symmetric Frenet curve of order $1$ is a \textit{semi-symmetric
geodesic}. A semi-symmetric Frenet curve of order $2$ with constant $%
\widetilde{\kappa }>0$ is a \textit{semi-symmetric circle}. A semi-symmetric
Frenet curve of order $3$ with constant $\widetilde{\kappa }>0$ and constant 
$\widetilde{\tau }>0$ is a \textit{semi-symmetric helix}.
\end{definition}

To sum up our construction, we call the next formulas as \textit{%
semi-symmetric Frenet formulas}:%
\begin{equation*}
\widetilde{\nabla }_{T}T=-\widetilde{\kappa }N,
\end{equation*}%
\begin{equation*}
\widetilde{\nabla }_{T}N=-\widetilde{\kappa }T+\widetilde{\tau }B,
\end{equation*}%
\begin{equation*}
\widetilde{\nabla }_{T}B=-\widetilde{\tau }N,
\end{equation*}%
where $\widetilde{\kappa }=\left\Vert \widetilde{\nabla }_{T}T\right\Vert $
is the semi-symmetric curvature and $\widetilde{\tau }=\left\Vert \widetilde{%
\nabla }_{T}N+\widetilde{\kappa }T\right\Vert $ is the semi-symmetric
torsion of $\gamma .$

\begin{remark}
Notice that semi-symmetric normal $N$ and semi-symmetric binormal $B$ do not
have to coincide with the ones of the Levi-Civita connection. But $T=\gamma
^{\prime }$ is the same. Also notice that for order $3$, $\left\{
T,N,B\right\} $ is fundamentally constructed using Gram-Schmidt process of 
\begin{equation*}
\left\{ T,\widetilde{\nabla }_{T}T,\widetilde{\nabla }_{T}\widetilde{\nabla }%
_{T}T\right\} ,
\end{equation*}%
which can also be calculated after construction as follows:

\begin{proposition}
For a semi-symmetric Frenet curve of order $3$,%
\begin{equation*}
\widetilde{\nabla }_{T}\widetilde{\nabla }_{T}T=-\widetilde{\kappa }^{2}T+%
\widetilde{\kappa }^{\prime }N+\widetilde{\kappa }\widetilde{\tau }B.
\end{equation*}%
For a semi-symmetric Frenet curve of order $2$,%
\begin{equation*}
\widetilde{\nabla }_{T}\widetilde{\nabla }_{T}T=-\widetilde{\kappa }^{2}T+%
\widetilde{\kappa }^{\prime }N.
\end{equation*}%
For a semi-symmetric Frenet curve of order $1$, $\widetilde{\nabla }_{T}%
\widetilde{\nabla }_{T}T=0.$
\end{proposition}
\end{remark}

\section{\label{sect-geodesicE3}Semi-Symmetric Geodesics of $\mathbb{E}^{3}$}

Let us consider the Euclidean space $\mathbb{E}^{3}$ endowed with the
semi-symmetric connection $\widetilde{\nabla }$, where we choose 
\begin{equation*}
U=\frac{\partial }{\partial z},\text{ }\omega \left( X\right) =g\left(
X,U\right) ,
\end{equation*}%
where $\left\{ x,y,z\right\} $ are coordinate functions, $%
g=dx^{2}+dy^{2}+dz^{2}$ is the Euclidean metric. It is well-known that 
\begin{equation*}
e_{1}=\frac{\partial }{\partial x},\text{ }e_{2}=\frac{\partial }{\partial y}%
,\text{ }e_{3}=\frac{\partial }{\partial z}
\end{equation*}%
form a g-orthonormal basis of $\chi \left( \mathbb{E}^{3}\right) $.

Let $\gamma :I\rightarrow \mathbb{E}^{3},$ $\gamma (s)=\left( \gamma
_{1}(s),\gamma _{2}(s),\gamma _{3}(s)\right) $ be a semi-symmetric geodesic
of $\mathbb{E}^{3}$, where $s$ denotes the arc-length parameter. Then we
find 
\begin{equation}
0=\widetilde{\nabla }_{T}T=\nabla _{T}T+g(T,U)T-g(T,T)U,  \label{ssmce3}
\end{equation}%
where $\nabla $ is the Levi-Civita connection and $T=\frac{d}{ds}=\gamma
^{\prime }=\left( \gamma _{1}^{\prime },\gamma _{2}^{\prime },\gamma
_{3}^{\prime }\right) $. In $\mathbb{E}^{3},$ it is well-known that 
\begin{equation*}
\nabla _{T}T=\gamma ^{\prime \prime }=\left( \gamma _{1}^{\prime \prime
},\gamma _{2}^{\prime \prime },\gamma _{3}^{\prime \prime }\right) .
\end{equation*}%
We also get 
\begin{equation*}
g(T,U)=\gamma _{3}^{\prime }
\end{equation*}%
and%
\begin{equation}
g(T,T)=\left( \gamma _{1}^{\prime }\right) ^{2}+\left( \gamma _{2}^{\prime
}\right) ^{2}+\left( \gamma _{3}^{\prime }\right) ^{2}=1.  \label{unit-E3}
\end{equation}%
From (\ref{ssmce3}), we can write%
\begin{equation*}
\left( \gamma _{1}^{\prime \prime },\gamma _{2}^{\prime \prime },\gamma
_{3}^{\prime \prime }\right) +\gamma _{3}^{\prime }.\left( \gamma
_{1}^{\prime },\gamma _{2}^{\prime },\gamma _{3}^{\prime }\right) -\left(
0,0,1\right) =0,
\end{equation*}%
which gives us the following system of differential equations: 
\begin{equation*}
\gamma _{1}^{\prime \prime }+\gamma _{1}^{\prime }\gamma _{3}^{\prime }=0,
\end{equation*}%
\begin{equation*}
\gamma _{2}^{\prime \prime }+\gamma _{2}^{\prime }\gamma _{3}^{\prime }=0,
\end{equation*}%
\begin{equation*}
\gamma _{3}^{\prime \prime }+\left( \gamma _{3}^{\prime }\right) ^{2}-1=0.
\end{equation*}%
If we solve this system, we can state the following result:

\begin{theorem}
Let $\gamma :I\rightarrow \mathbb{E}^{3},$ $\gamma (s)=\left( \gamma
_{1}(s),\gamma _{2}(s),\gamma _{3}(s)\right) $ be a semi-symmetric Frenet
curve of $\mathbb{E}^{3}$ endowed with the semi-symmetric connection $%
\widetilde{\nabla }$ with unit vector field $U=\frac{\partial }{\partial z}$%
, where $s$ denotes the arc-length parameter. Then $\gamma $ is a
semi-symmetric geodesic if and only if%
\begin{equation*}
\gamma _{1}(s)=e^{-c_{1}}.c_{2}.\arctan \left( e^{s-c_{1}}\right) +c_{3},
\end{equation*}%
\begin{equation*}
\gamma _{2}(s)=e^{-c_{1}}.c_{4}.\arctan \left( e^{s-c_{1}}\right) +c_{5},
\end{equation*}%
\begin{equation*}
\gamma _{3}(s)=-s+\ln \left( e^{2s}+e^{2c_{1}}\right) +c_{6},
\end{equation*}%
where $c_{i}$, $i=\overline{1,6}$ are arbitrary constants satisfying 
\begin{equation*}
4e^{2c_{1}}=c_{2}^{2}+c_{4}^{2}.
\end{equation*}
\end{theorem}

\section{\label{sect-geodesic-R3(-3)}Semi-symmetric Geodesics of $\mathbb{R}%
^{3}(-3)$}

Let $M=\mathbb{R}^{3}$ with coordinate functions $\left\{ x,y,z\right\} $.
Let us define%
\begin{equation*}
\xi =2\frac{\partial }{\partial z},\text{ }\eta =\frac{1}{2}\left(
dz-ydx\right) ,
\end{equation*}%
\begin{equation*}
X=2\frac{\partial }{\partial y},\text{ }Y=2\left( \frac{\partial }{\partial x%
}+y\frac{\partial }{\partial z}\right) ,
\end{equation*}%
\begin{equation*}
\varphi X=Y,\text{ }\varphi \xi =0,
\end{equation*}%
\begin{equation*}
g=\frac{1}{4}\left( dx^{2}+dy^{2}\right) +\eta \otimes \eta .
\end{equation*}%
$\left( M,\varphi ,\xi ,\eta ,g\right) $ is a Sasakian manifold known as $%
\mathbb{R}^{3}(-3)$. $\left\{ X,Y,\xi \right\} $ is a $g$-orthonormal basis
of $\chi \left( M\right) $ and the Levi-Civita connection is calculated as%
\begin{equation*}
\nabla _{X}X=0,\text{ }\nabla _{X}Y=\xi ,\nabla _{X}\xi =-Y,
\end{equation*}%
\begin{equation*}
\nabla _{Y}X=-\xi ,\text{ }\nabla _{Y}Y=0,\nabla _{Y}\xi =X,
\end{equation*}%
\begin{equation*}
\nabla _{\xi }X=-Y,\text{ }\nabla _{\xi }Y=X,\nabla _{\xi }\xi =0.
\end{equation*}%
If we choose $U=\xi $ and denote the semi-symmetric metric connection by $%
\widetilde{\nabla }$, we find%
\begin{equation*}
\widetilde{\nabla }_{X}X=-\xi ,\text{ }\widetilde{\nabla }_{X}Y=\xi ,%
\widetilde{\nabla }_{X}\xi =X-Y,
\end{equation*}%
\begin{equation*}
\widetilde{\nabla }_{Y}X=-\xi ,\text{ }\widetilde{\nabla }_{Y}Y=-\xi ,%
\widetilde{\nabla }_{Y}\xi =X+Y,
\end{equation*}%
\begin{equation*}
\widetilde{\nabla }_{\xi }X=-Y,\text{ }\widetilde{\nabla }_{\xi }Y=X,%
\widetilde{\nabla }_{\xi }\xi =0.
\end{equation*}%
Let $\gamma :I\rightarrow \mathbb{R}^{3}(-3),$ $\gamma =\left( \gamma
_{1},\gamma _{2},\gamma _{3}\right) $ be a semi-symmetric geodesic. Then%
\begin{eqnarray*}
T &=&\left( \gamma _{1}^{\prime },\gamma _{2}^{\prime },\gamma _{3}^{\prime
}\right) \\
&=&\frac{1}{2}\left[ \gamma _{2}^{\prime }X+\gamma _{1}^{\prime }Y+\left(
\gamma _{3}^{\prime }-\gamma _{1}^{\prime }\gamma _{2}\right) \xi \right] .
\end{eqnarray*}%
Since $\gamma $ is unit-speed, we have%
\begin{equation}
\left( \gamma _{1}^{\prime }\right) ^{2}+\left( \gamma _{2}^{\prime }\right)
^{2}+\left( \gamma _{3}^{\prime }-\gamma _{1}^{\prime }\gamma _{2}\right)
^{2}=4.  \label{unitR3}
\end{equation}%
Using $\widetilde{\nabla }$, we obtain%
\begin{equation}
\widetilde{\nabla }_{T}X=\frac{1}{2}\left[ -\left( \gamma _{1}^{\prime
}+\gamma _{2}^{\prime }\right) \xi -\left( \gamma _{3}^{\prime }-\gamma
_{1}^{\prime }\gamma _{2}\right) Y\right] ,  \label{*}
\end{equation}%
\begin{equation}
\widetilde{\nabla }_{T}Y=\frac{1}{2}\left[ \left( \gamma _{2}^{\prime
}-\gamma _{1}^{\prime }\right) \xi +\left( \gamma _{3}^{\prime }-\gamma
_{1}^{\prime }\gamma _{2}\right) X\right] ,  \label{**}
\end{equation}%
\begin{equation}
\widetilde{\nabla }_{T}\xi =\frac{1}{2}\left[ \left( \gamma _{1}^{\prime
}+\gamma _{2}^{\prime }\right) X+\left( \gamma _{2}^{\prime }-\gamma
_{1}^{\prime }\right) Y\right] .  \label{***}
\end{equation}%
As a result, we get%
\begin{equation}
\widetilde{\nabla }_{T}T=\frac{1}{2}\left[ \lambda X+\mu Y+\nu \xi \right]
=0,  \label{geodesicR3}
\end{equation}%
where%
\begin{equation*}
\lambda =\gamma _{2}^{\prime \prime }-\frac{1}{2}\gamma _{2}\gamma
_{1}^{\prime }\gamma _{2}^{\prime }-\gamma _{2}\left( \gamma _{1}^{\prime
}\right) ^{2}+\gamma _{1}^{\prime }\gamma _{3}^{\prime }+\frac{1}{2}\gamma
_{2}^{\prime }\gamma _{3}^{\prime },
\end{equation*}%
\begin{equation*}
\mu =\gamma _{1}^{\prime \prime }+\gamma _{2}\gamma _{1}^{\prime }\gamma
_{2}^{\prime }-\frac{1}{2}\gamma _{2}\left( \gamma _{1}^{\prime }\right)
^{2}+\frac{1}{2}\gamma _{1}^{\prime }\gamma _{3}^{\prime }-\gamma
_{2}^{\prime }\gamma _{3}^{\prime },
\end{equation*}%
\begin{equation*}
\nu =-\frac{1}{2}\left( \gamma _{1}^{\prime }\right) ^{2}-\frac{1}{2}\left(
\gamma _{1}^{\prime }\right) ^{2}+\gamma _{3}^{\prime \prime }-\gamma
_{1}^{\prime \prime }\gamma _{2}-\gamma _{1}^{\prime }\gamma _{2}^{\prime }.
\end{equation*}%
Now we can state the following theorem:

\begin{theorem}
Let $\gamma :I\rightarrow \mathbb{R}^{3}(-3),$ $\gamma (s)=\left( \gamma
_{1}(s),\gamma _{2}(s),\gamma _{3}(s)\right) $ be a semi-symmetric Frenet
curve of $\mathbb{R}^{3}(-3)$ endowed with the semi-symmetric connection $%
\widetilde{\nabla }$ with unit vector field $U=\xi $, where $s$ denotes the
arc-length parameter. Then $\gamma $ is a semi-symmetric geodesic if and
only if it satisfies the system of non-linear differential equations $%
\lambda =0,$ $\mu =0,\nu =0,$ and 
\begin{equation*}
\left( \gamma _{1}^{\prime }\right) ^{2}+\left( \gamma _{2}^{\prime }\right)
^{2}+\left( \gamma _{3}^{\prime }-\gamma _{1}^{\prime }\gamma _{2}\right)
^{2}=4.
\end{equation*}
\end{theorem}

\begin{corollary}
\textbf{\ }If $\gamma :I\rightarrow \mathbb{R}^{3}(-3)$ is a semi-symmetric
geodesic, then%
\begin{equation*}
\gamma _{3}^{\prime }-\gamma _{1}^{\prime }\gamma _{2}=\frac{2\left(
e^{2s}-e^{4c_{1}}\right) }{e^{2s}+e^{4c_{1}}},
\end{equation*}%
\begin{equation*}
\left( \gamma _{1}^{\prime }\right) ^{2}+\left( \gamma _{2}^{\prime }\right)
^{2}=\frac{16\left( e^{2s+4c_{1}}\right) }{e^{2s}+e^{4c_{1}}},
\end{equation*}%
where $s$ is the arc-length parameter and $c_{1}$ is a constant.
\end{corollary}

\begin{proof}
From (\ref{unitR3}) and $\nu =0$, we find%
\begin{equation*}
2\left( \gamma _{3}^{\prime }-\gamma _{1}^{\prime }\gamma _{2}\right)
^{\prime }=\left( \gamma _{1}^{\prime }\right) ^{2}+\left( \gamma
_{2}^{\prime }\right) ^{2}=4-\left( \gamma _{3}^{\prime }-\gamma
_{1}^{\prime }\gamma _{2}\right) ^{2}.
\end{equation*}%
If we denote $f(s)=\gamma _{3}^{\prime }-\gamma _{1}^{\prime }\gamma _{2},$
we have Riccati's differential equation%
\begin{equation*}
2f^{\prime }=4-f^{2},
\end{equation*}%
which gives the result.
\end{proof}

\section{\label{sect-geodesic-H3(-1)}Semi-symmetric Geodesics of $\mathbb{H}%
^{3}(-1)$}

Let us recall the Poincare model of the hyperbolic space. Let 
\begin{equation*}
M=\left\{ \left( x,y,z\right) \in \mathbb{R}^{3}:z>0\right\} ,
\end{equation*}%
\begin{equation*}
g=\frac{1}{z^{2}}\left( dx^{2}+dy^{2}+dz^{2}\right) .
\end{equation*}%
Then 
\begin{equation*}
e_{1}=z\frac{\partial }{\partial x},\text{ }e_{2}=z\frac{\partial }{\partial
y},\text{ }e_{3}=-z\frac{\partial }{\partial z}
\end{equation*}%
form a $g$-orthonormal basis of $\mathbb{H}^{3}(-1)$. The corresponding
Levi-Civita connection is%
\begin{equation*}
\nabla _{e_{1}}e_{1}=-e_{3},\text{ }\nabla _{e_{1}}e_{2}=0,\text{ }\nabla
_{e_{1}}e_{3}=e_{1},
\end{equation*}%
\begin{equation*}
\nabla _{e_{2}}e_{1}=0,\text{ }\nabla _{e_{2}}e_{2}=-e_{3},\text{ }\nabla
_{e_{2}}e_{3}=e_{2},
\end{equation*}%
\begin{equation*}
\nabla _{e_{3}}e_{1}=0,\text{ }\nabla _{e_{3}}e_{2}=0,\text{ }\nabla
_{e_{3}}e_{3}=0.
\end{equation*}%
If we choose $U=e_{3}$, we calculate the semi-symmetric metric connection $%
\widetilde{\nabla }$ as%
\begin{equation*}
\widetilde{\nabla }_{e_{1}}e_{1}=-2e_{3},\text{ }\widetilde{\nabla }%
_{e_{1}}e_{2}=0,\text{ }\widetilde{\nabla }_{e_{1}}e_{3}=2e_{1},
\end{equation*}%
\begin{equation*}
\widetilde{\nabla }_{e_{2}}e_{1}=0,\text{ }\widetilde{\nabla }%
_{e_{2}}e_{2}=-2e_{3},\text{ }\widetilde{\nabla }_{e_{2}}e_{3}=2e_{2},
\end{equation*}%
\begin{equation*}
\widetilde{\nabla }_{e_{3}}e_{1}=0,\text{ }\widetilde{\nabla }%
_{e_{3}}e_{2}=0,\text{ }\widetilde{\nabla }_{e_{3}}e_{3}=0.
\end{equation*}

Let $\gamma :I\rightarrow H^{3}\left( -1\right) $ be a semi-symmetric
geodesic. Then%
\begin{eqnarray*}
T &=&\gamma ^{\prime }=\left( \gamma _{1}^{\prime },\gamma _{2}^{\prime
},\gamma _{3}^{\prime }\right)  \\
&=&\frac{\gamma _{1}^{\prime }}{\gamma _{3}}e_{1}+\frac{\gamma _{2}^{\prime }%
}{\gamma _{3}}e_{2}-\frac{\gamma _{3}^{\prime }}{\gamma _{3}}e_{3}.
\end{eqnarray*}%
Since $g\left( T,T\right) =1$, we find%
\begin{equation}
\left( \gamma _{1}^{\prime }\right) ^{2}+\left( \gamma _{2}^{\prime }\right)
^{2}+\left( \gamma _{3}^{\prime }\right) ^{2}=\gamma _{3}^{2}.
\label{unitH3}
\end{equation}%
Using $\widetilde{\nabla }$, we obtain%
\begin{equation*}
\widetilde{\nabla }_{T}e_{1}=-\frac{2\gamma _{1}^{\prime }}{\gamma _{3}}%
e_{3},
\end{equation*}%
\begin{equation*}
\widetilde{\nabla }_{T}e_{2}=-\frac{2\gamma _{2}^{\prime }}{\gamma _{3}}%
e_{3},
\end{equation*}%
\begin{equation*}
\widetilde{\nabla }_{T}e_{3}=\frac{2\gamma _{1}^{\prime }}{\gamma _{3}}e_{1}+%
\frac{2\gamma _{2}^{\prime }}{\gamma _{3}}e_{2}.
\end{equation*}%
As a result, we find%
\begin{eqnarray*}
\widetilde{\nabla }_{T}T &=&\frac{1}{\gamma _{3}^{2}}\left[ \left( \gamma
_{1}^{\prime \prime }\gamma _{3}-3\gamma _{1}^{\prime }\gamma _{3}^{\prime
}\right) e_{1}\right.  \\
&&+\left( \gamma _{2}^{\prime \prime }\gamma _{3}-3\gamma _{2}^{\prime
}\gamma _{3}^{\prime }\right) e_{2} \\
&&\left. -\left( \gamma _{3}^{\prime \prime }\gamma _{3}-\left( \gamma
_{3}^{\prime }\right) ^{2}+2\left( \gamma _{1}^{\prime }\right) ^{2}+2\left(
\gamma _{2}^{\prime }\right) ^{2}\right) e_{3}\right]  \\
&=&0.
\end{eqnarray*}%
Now we can state the following theorem:

\begin{theorem}
Let $\gamma :I\rightarrow \mathbb{H}^{3}(-1),$ $\gamma (s)=\left( \gamma
_{1}(s),\gamma _{2}(s),\gamma _{3}(s)\right) $ be a semi-symmetric Frenet
curve of $\mathbb{H}^{3}(-1)$ endowed with the semi-symmetric connection $%
\widetilde{\nabla }$ with unit vector field $U=e_{3}$, where $s$ denotes the
arc-length parameter. Then $\gamma $ is a semi-symmetric geodesic if and
only if 
\begin{equation*}
\gamma _{1}(s)=k_{1}+\frac{2c_{2}e^{3s+\frac{c_{3}}{2}}}{3\sqrt{1+c_{1}^{2}}}%
\text{ }_{2}F_{1}\left( \frac{3}{4},\frac{3}{2},\frac{7}{4}%
;-e^{4s+c_{3}}\right) ,
\end{equation*}%
\begin{equation*}
\gamma _{2}(s)=l_{1}+\frac{2c_{2}e^{3s+\frac{c_{3}}{2}}}{3\sqrt{1+c_{1}^{2}}}%
\text{ }_{2}F_{1}\left( \frac{3}{4},\frac{3}{2},\frac{7}{4}%
;-e^{4s+c_{3}}\right) ,
\end{equation*}%
\begin{equation*}
\gamma _{3}(s)=\frac{c_{2}e^{s}}{\sqrt{e^{c_{3}+4s}+1}},
\end{equation*}%
where $c_{1,2,3}$ and $k_{1},l_{1}$ are constants $\left( c_{2}>0\right) $
and $_{2}F_{1}$ is the hypergeometric function.
\end{theorem}

\section{\label{sect-examples}Examples of Semi-Symmetric Frenet Curves}

Next examples are based on the choice of $U$ as previous sections. Firstly,
we give examples in $\mathbb{E}^{3}$:

\begin{example}
$\gamma :I\rightarrow \mathbb{E}^{3},$ $\gamma \left( s\right) =\left(
s,0,1\right) $ is a semi-symmetric circle with $\widetilde{\kappa }=1$ and
the semi-symmetric Frenet frame%
\begin{equation*}
\left\{ T=\left( 1,0,0\right) ,\text{ }N=(0,0,-1)\right\} .
\end{equation*}%
It is obvious that it is a line (i.e. a geosedic) in the usual sense with no
curvature and torsion. It only has a tangential vector field $T=\left(
1,0,0\right) .$
\end{example}

\begin{example}
$\gamma :I\rightarrow \mathbb{E}^{3},$ $\gamma \left( s\right) =\left(
0,s,1\right) $ is also a semi-symmetric circle with $\widetilde{\kappa }=1.$
It is a line as well in the usual sense.
\end{example}

\begin{example}
$\gamma :I\rightarrow \mathbb{E}^{3},$ $\gamma \left( s\right) =\left( \cos
s,\sin s,0\right) $ is a semi-symmetric circle with $\widetilde{\kappa }=%
\sqrt{2}$ and the semi-symmetric Frenet frame field%
\begin{equation*}
\left\{ T=\left( -\sin s,\cos s,0\right) ,\text{ }N=\frac{1}{\sqrt{2}}\left(
-\cos s,-\sin s,1\right) \right\} .
\end{equation*}%
It is a circle with curvature $1$ in the usual sense.
\end{example}

\begin{example}
$\gamma :I\rightarrow \mathbb{E}^{3},$ $\gamma \left( s\right) =\left(
0,\cos s,\sin s\right) $ is a semi-symmetric Frenet curve of order $2$ with $%
\widetilde{\kappa }=1+\sin s$ and the Frenet frame field%
\begin{equation*}
\left\{ T=\left( 0,-\sin s,\cos s\right) ,\text{ }N=\left( 0,-\cos s,-\sin
s\right) \right\} .
\end{equation*}%
It is a circle with curvature $1$ in the usual sense. Notice that it is 
\textbf{not} a semi-symmetric circle.
\end{example}

\begin{example}
$\gamma :I\rightarrow \mathbb{E}^{3},$ $\gamma \left( s\right) =\left(
\gamma _{1}(s),\gamma _{2}(s),\gamma _{3}(s)\right) $ be a unit-speed
semi-symmetric Frenet curve of order $3$. Then 
\begin{equation*}
\left( \gamma _{1}^{\prime }\right) ^{2}+\left( \gamma _{2}^{\prime }\right)
^{2}+\left( \gamma _{3}^{\prime }\right) ^{2}=1
\end{equation*}%
and one can calculate%
\begin{equation*}
\widetilde{\kappa }=\sqrt{\left( \gamma _{1}^{\prime \prime }+\gamma
_{1}^{\prime }\gamma _{3}^{\prime }\right) ^{2}+\left( \gamma _{2}^{\prime
\prime }+\gamma _{2}^{\prime }\gamma _{3}^{\prime }\right) ^{2}+\left(
\gamma _{3}^{\prime \prime }+\left( \gamma _{3}^{\prime }\right)
^{2}-1\right) ^{2}},
\end{equation*}%
\begin{equation*}
\widetilde{\tau }=\frac{1}{\widetilde{\kappa }}\left\Vert \widetilde{\nabla }%
_{T}\widetilde{\nabla }_{T}T+\widetilde{\kappa }^{2}T-\widetilde{\kappa }%
^{\prime }N\right\Vert .
\end{equation*}
\end{example}

Secondly, we give examples in $\mathbb{R}^{3}(-3)$:

\begin{example}
$\gamma :I\rightarrow \mathbb{R}^{3}(-3),$ $\gamma \left( s\right) =\left(
0,2s,1\right) $ is a semi-symmetric helix with $\widetilde{\kappa }=%
\widetilde{\tau }=1$ and the semi-symmetric Frenet vector fields%
\begin{equation*}
T=X,
\end{equation*}%
\begin{equation*}
N=-\xi ,
\end{equation*}%
\begin{equation*}
B=Y.
\end{equation*}
\end{example}

\begin{example}
$\gamma :I\rightarrow \mathbb{R}^{3}(-3),$ $\gamma \left( s\right) =\left(
2s,0,1\right) $ is another semi-symmetric helix with $\widetilde{\kappa }=%
\widetilde{\tau }=1$ and the semi-symmetric Frenet vector fields%
\begin{equation*}
T=Y,
\end{equation*}%
\begin{equation*}
N=-\xi ,
\end{equation*}%
\begin{equation*}
B=-X.
\end{equation*}
\end{example}

\begin{example}
$\gamma :I\rightarrow \mathbb{R}^{3}(-3),$ $\gamma \left( s\right) =\left(
2\cos s,0,2\sin s\right) $ is \textbf{not} a semi-symmetric circle since it
has semi-symmetric curvature 
\begin{equation*}
\widetilde{\kappa }=\sqrt{4\sin ^{2}s\cos ^{2}s+\left( 1+\sin s\right) ^{2}}.
\end{equation*}
\end{example}

Finally an example will be given in $\mathbb{H}^{3}\left( -1\right) $ using%
\begin{equation*}
T=\frac{\gamma _{1}^{\prime }}{\gamma _{3}}e_{1}+\frac{\gamma _{2}^{\prime }%
}{\gamma _{3}}e_{2}-\frac{\gamma _{3}^{\prime }}{\gamma _{3}}e_{3},
\end{equation*}%
\begin{equation*}
\left( \gamma _{1}^{\prime }\right) ^{2}+\left( \gamma _{2}^{\prime }\right)
^{2}+\left( \gamma _{3}^{\prime }\right) ^{2}=\gamma _{3}.
\end{equation*}
\begin{eqnarray*}
\widetilde{\nabla }_{T}T &=&\frac{1}{\gamma _{3}^{2}}\left[ \left( \gamma
_{1}^{\prime \prime }\gamma _{3}-3\gamma _{1}^{\prime }\gamma _{3}^{\prime
}\right) e_{1}\right. \\
&&+\left( \gamma _{2}^{\prime \prime }\gamma _{3}-3\gamma _{2}^{\prime
}\gamma _{3}^{\prime }\right) e_{2} \\
&&\left. -\left( \gamma _{3}^{\prime \prime }\gamma _{3}-\left( \gamma
_{3}^{\prime }\right) ^{2}+2\left( \gamma _{1}^{\prime }\right) ^{2}+2\left(
\gamma _{2}^{\prime }\right) ^{2}\right) e_{3}\right] .
\end{eqnarray*}

\begin{example}
$\gamma :I\rightarrow \mathbb{H}^{3}(-1),$ $\gamma \left( s\right) =\left(
s,0,1\right) $ is a semi-symmetric circle with $T=e_{1},$ $N=-e_{3},$ $%
\widetilde{\kappa }=2.$
\end{example}

\end{document}